\documentclass[a4paper,10pt]{amsart}
\usepackage{microtype}
\usepackage[utf8]{inputenc}
\usepackage{amsfonts, amsmath, amssymb, amsthm, mathtools}
\usepackage{xcolor}
\definecolor{darkblue}{rgb}{0,0,0.6}
\usepackage{tikz}
\usetikzlibrary{matrix,arrows,trees}
\usepackage{tikz-cd}
\usepackage[ocgcolorlinks,colorlinks=true, citecolor=darkblue, filecolor=darkblue, linkcolor=darkblue, urlcolor=darkblue]{hyperref}
\usepackage[nameinlink, capitalise]{cleveref}

\crefname{lem}{Lemma}{Lemmas}
\crefname{prop}{Proposition}{Propositions}
\crefname{rem}{Remark}{Remarks}

\DeclareRobustCommand{\SkipTocEntry}[5]{}

\usepackage{eucal}                  

\hypersetup{bookmarksopen=true}

\newcommand*\cocolon{%
	\nobreak
	\mskip6mu plus1mu
	\mathpunct{}%
	\nonscript
	\mkern-\thinmuskip
	{:}%
	\mskip2mu
	\relax
}

\numberwithin{equation}{section}

\swapnumbers
\theoremstyle{definition}
\newtheorem{defn}[equation]{Definition}
\newtheorem{ex}[equation]{Example}
\newtheorem{const}[equation]{Construction}
\newtheorem{rem}[equation]{Remark}
\newtheorem{theorem}[equation]{Theorem}
\newtheorem{prop}[equation]{Proposition}
\newtheorem{lem}[equation]{Lemma}
\newtheorem{cor}[equation]{Corollary}

\renewcommand{\epsilon}{\varepsilon}
\renewcommand{\theta}{\vartheta}
\renewcommand{\phi}{\varphi}

\renewcommand{\tilde}{\widetilde}

\newcommand{\cA}{\mathcal{A}}
\newcommand{\cB}{\mathcal{B}}
\newcommand{\cC}{\mathcal{C}}
\newcommand{\cD}{\mathcal{D}}

\newcommand{\cK}{\mathcal{K}}

\newcommand{\cO}{\mathcal{O}}
\newcommand{\cP}{\mathcal{P}}

\newcommand{\cU}{\mathcal{U}}

\newcommand{\cX}{\mathcal{X}}
\newcommand{\cY}{\mathcal{Y}}

\DeclareMathOperator{\Alg}{Alg}
\DeclareMathOperator{\Ar}{Ar}

\DeclareMathOperator{\con}{const}
\DeclareMathOperator{\Nat}{Nat}

\DeclareMathOperator{\St}{St}

\DeclareMathOperator{\Un}{Un}

\DeclareMathOperator{\Fun}{Fun}

\DeclareMathOperator{\Hom}{Hom}
\DeclareMathOperator{\id}{id}
\DeclareMathOperator{\inc}{inc}

\DeclareMathOperator{\Mon}{Mon}

\newcommand{\st}{\mathrm{st}}

\newcommand{\Cart}{\mathrm{Cart}}
\newcommand{\cat}{\mathrm{Cat}}
\newcommand{\catinf}{\mathrm{Cat}}

\newcommand{\catst}{\catinf^{\st}}

\newcommand{\Cocart}{\mathrm{Cocart}}

\newcommand{\LFib}{\mathrm{LFib}}
\newcommand{\RFib}{\mathrm{RFib}}

\newcommand{\op}{\mathrm{op}}

\newcommand{\Spc}{\mathrm{An}}

\newcounter{commentcounter}

\newcommand{\Comm}{\mathrm{Comm}}
\newcommand{\Bifib}{\mathrm{Bifib}}
\newcommand{\CurvOrtho}{\mathrm{CurvOrtho}}
\newcommand{\Ortho}{\mathrm{Ortho}}
\newcommand{\bicart}{\mathrm{bicart}}
\newcommand{\cart}{\mathrm{cart}}
\newcommand{\cocart}{\mathrm{cocart}}
\newcommand{\Day}{\mathrm{Day}}
\newcommand{\lax}{\mathrm{lax}}
\newcommand{\lslice}[3]{#1_{#2 /\hspace{-.5ex}/ #3}}
\newcommand{\Op}{\mathrm{Op}}
\newcommand{\oplax}{\mathrm{opl}}

\DeclareMathOperator{\ev}{ev}
\DeclareMathOperator{\Fr}{Fr}

\DeclareFontFamily{U}{min}{}
\DeclareFontShape{U}{min}{m}{n}{<-> dmjhira}{}
\newcommand{\yo}{\text{\usefont{U}{min}{m}{n}\symbol{'110}}}

\title{Hinich's model for Day convolution revisited}

\author{Christoph Winges}
\address{Fakult\"at f\"ur Mathematik, Universit\"at Regensburg, 93040 Regensburg, Germany}
\email{christoph.winges@ur.de}

\date{}

\begin{document}

\begin{abstract}
    We prove that Hinich's construction of the Day convolution operad of two $\cO$-monoidal $\infty$-categories is an exponential in the $\infty$-category of $\infty$-operads over $\cO$, and use this to give an explicit description of the formation of algebras in the Day convolution operad as a bivariant functor.
\end{abstract}

\maketitle

\section{Introduction}

Let $\cO^\otimes$ be an $\infty$-operad and consider operad maps $p \colon \cC^\otimes \to \cO^\otimes$ and $q \colon \cD^\otimes \to \cO^\otimes$.
If it exists, the \emph{Day convolution} of $p$ and $q$ is an exponential of $q$ by $p$ in the $\infty$-category of $\infty$-operads over $\cO^\otimes$, ie a right adjoint object to $q$ with respect to the product functor $- \times_{\cO^\otimes} \cC^\otimes \colon (\Op_\infty)_{/\cO^\otimes} \to (\Op_\infty)_{/\cO^\otimes}$.
Using the combinatorial machinery of quasicategories, Lurie gives a very general construction of Day convolution operads in \cite[Section~2.2.6]{HA};
see also \cite[Section~2.8]{hinich:enriched-yoneda}.

For practical purposes, it is often sufficient to know that the Day convolution operad exists in the case that $\cC^\otimes$ and $\cD^\otimes$ are $\cO$-monoidal categories.
In the case of symmetric monoidal $\infty$-categories, Glasman gives a concrete description of the Day convolution operad in \cite{glasman:day}.
Hinich provides in \cite{hinich:enriched-yoneda} a rather straightforward description of the Day convolution operad for arbitrary $\cO$-monoidal $\infty$-categories, and it is this model we will focus on in the sequel.

To facilitate Hinich's description, recall that there are equivalences of $\infty$-categories
\[ \Alg_\cO(\catinf_\infty) \xrightarrow{\sim} \Mon_\cO(\cat_\infty) \simeq \Cocart_\otimes(\cO). \]
Here, $\Cocart_\otimes(\cO)$ denotes the $\infty$-category of cocartesian fibrations of $\infty$-operads over $\cO^\otimes$.
The first equivalence is given by composition with the cartesian structure $\cat_\infty^\times \to \cat_\infty$ \cite[Proposition~2.4.2.5]{HA} and the second equivalence is induced by (un)straightening; this is implicit in \cite{HA} and explicitly spelled out in \cite[Proposition~A.2.1]{hinich:rectification}.
In the sequel, we will freely switch between these descriptions of $\cO$-monoidal $\infty$-categories as necessary.

Denote by $\Ar^\oplax$ the full subcategory of $(\cat_\infty)_{/[1]}$ spanned by the cartesian fibrations and equip it with the cartesian symmetric monoidal structure.
Restriction to $\{0\}$ and $\{1\}$ defines symmetric monoidal functors $t \colon \Ar^\oplax \to \cat_\infty$ and $s \colon \Ar^\oplax \to \cat_\infty$, where we also equip $\cat_\infty$ with the cartesian symmetric monoidal structure.

\begin{theorem}[{\cite[Section~2.8.9]{hinich:enriched-yoneda}}]\label{thm:day}
    Let $\cC$ and $\cD$ be $\cO$-monoidal $\infty$-categories and denote by $\Day_{\cC,\cD}^\otimes$ the pullback
    \[\begin{tikzcd}
        \Day_{\cC,\cD}^\otimes\ar[r]\ar[d] & (\Ar^\oplax)^\times\ar[d, "{(s,t)}"] \\
        \cO^\otimes\ar[r, "{(\cC,\cD)}"] & \cat_\infty^\times \times_{\Comm^\otimes} \cat_\infty^\times
    \end{tikzcd}\]
    Then $\Day_{\cC,\cD}^\otimes \to \cO^\otimes$ is the Day convolution of $\cC^\otimes$ and $\cD^\otimes$.
\end{theorem}

The goal of this note is to give an alternative proof of this theorem.
Instead of identifying $\Day_{\cC,\cD}^\otimes$ with another model for the Day convolution operad as in \cite[Section~6.3.9]{hinich:enriched-yoneda}, we verify directly that $\Day_{\cC,\cD}^\otimes$ possesses the correct universal property, also on the level of categories of algebras.
In fact, we show first that the construction of $\Day_{\cC,\cD}^\otimes$ promotes the assignment $(\cC,\cD) \mapsto \Alg_{/\cO}(\Day_{\cC,\cD})$ to a bivariant functor and exhibit a natural equivalence
\[ \Alg_{/\cO}(\Day_{\cC,\cD}) \simeq \Alg_{\cC/\cO}(\cD), \]
where the functoriality of the right hand side is given by pre- and postcomposition with operad maps.
The universal property of $\Day_{\cC,\cD}^\otimes$ follows easily from this.
This is done in \cref{sec:main}.

\cref{sec:variations} comments on some variations of these statements for $\cO$-monoidal $\infty$-categories living in certain suboperads of the cartesian symmetric monoidal structure on $\cat_\infty$.
\cref{sec:day-o-monoidal} reproves the well-known statement that $\Day_{\cC,\cD}^\otimes \to \cO^\otimes$ defines an $\cO$-monoidal $\infty$-category under suitable cocompleteness assumptions on $\cD$.


\addtocontents{toc}{\SkipTocEntry}
\subsection*{Conventions}
\begin{enumerate}
    \item In the remainder of this note, the word ``category'' means ``$\infty$-category''.
    We write $\catinf$ for the category of small categories.
    \item The category of anima/spaces/$\infty$-groupoids is denoted by $\Spc$.
    The groupoid core $\iota \colon \cat \to \Spc$ is the right adjoint to the inclusion of $\Spc$ into $\catinf$.
    \item Given a category $X$, we denote by $\Cocart(X)$ the subcategory of $\catinf_{/X}$ given by the cocartesian fibrations over $X$ and those functors over $X$ which preserve cocartesian morphisms.
    \item We also drop the prefix ``$\infty$-'' from related concepts.
    For example, ``operad'' means ``$\infty$-operad'' from now on, and we denote the category of operads by $\Op$.
    \item The notation concerning operads and categories of algebras follows the conventions of \cite{HA}. The symbol $\Comm^\otimes$ denotes the commutative operad (ie the category of pointed finite sets).
    \item Given an operad $\cA^\otimes$, we denote by $\Cocart_\otimes(\cA)$ the subcategory of $\Op_{/\cA}$ given by the cocartesian fibrations of operads over $\cA$, ie $\cA$-monoidal categories, and those functors over $\cA$ which preserve cocartesian morphisms (corresponding to $\cA$-monoidal functors).
\end{enumerate}

\addtocontents{toc}{\SkipTocEntry}
\subsection*{Acknowledgements}
I am grateful to Bastiaan Cnossen, Fabian Hebestreit and Sil Linskens for discussions and comments on earlier versions of this document.

The author was supported by CRC 1085 ``Higher Invariants" funded by the Deutsche Forschungsgemeinschaft (DFG).


\section{Algebras in the Day convolution operad}\label{sec:main}

Throughout this section, fix a base operad $\cO^\otimes$ as well as $\cO$-monoidal categories $\cC$ and $\cD$.
Our goal is to prove the following strengthening of \cref{thm:day}.

\begin{theorem}\label{thm:day-improved}
    Let $\alpha \colon \cA^\otimes \to \cO^\otimes$ be an operad over $\cO^\otimes$.
    Then there exists a natural equivalence
    \[ \Alg_{\cA \times_\cO \cC/\cO}(\cD) \simeq \Alg_{\cA/\cO}(\Day_{\cC,\cD}). \]
\end{theorem}

Note that \cref{thm:day} follows from this by passing to groupoid cores.
The proof of this theorem is already implicit in \cite[Remark~5.2.5]{thenine-1}, but 
it will be convenient to formulate the argument in terms of some concepts introduced in
\cite{hhln:lax-adjunctions}.

\begin{defn}[{\cite[Observation~2.3.2]{hhln:lax-adjunctions}}]
 A functor $(p_1,p_2) \colon X \to Y \times Z$ is a \emph{curved orthofibration} if
 \begin{enumerate}
  \item $p_1 \colon X \to Y$ is a cartesian fibration;
  \item $p_2 \colon X \to Z$ is a cocartesian fibration;
  \item $p_1$-cartesian lifts of morphisms project to equivalences under $p_2$;
  \item $p_2$-cocartesian lifts of morphisms project to equivalences under $p_1$.
 \end{enumerate}
\end{defn}

Denote by $\Cocart^\lax(Y)$ the full subcategory of $\catinf_{/Y}$ spanned by the cocartesian fibrations, and by $\Cart^\oplax(Z)$ the full subcategory of $\catinf_{/Z}$ spanned by the cartesian fibrations.
In addition, $\CurvOrtho(Y,Z)$ denotes the subcategory of $\catinf_{/Y \times Z}$ whose objects are curved orthofibrations and whose morphisms are functors over $Y \times Z$ preserving both $p_1$-cartesian and $p_2$-cocartesian morphisms.
We will make use of the following description of $\CurvOrtho(Y,Z)$.

\begin{prop}[{\cite[Corollary~2.3.4]{hhln:lax-adjunctions}}]\label{prop:straighten-curved-ortho}
    Unstraightening over $Y$ and $Z$ induces equivalences
    \[ \Fun(Y^\op,\Cocart^\lax(Z))^\cart \simeq \CurvOrtho(Y,Z) \simeq \Fun(Z,\Cart^\oplax(Y))^\cocart \]
    which are natural in $Y$ and $Z$; the superscripts $\cocart$ and $\cart$ denote the wide subcategories on those natural transformations whose components all preserve (co)cartesian morphisms.
\end{prop}

We will also require the following statement about the interaction of cartesian and cocartesian morphisms in curved orthofibrations.

\begin{lem}\label{lem:preserve-lifts}
    Let $(p_1,p_2) \colon X \to Y \times Z$ be a curved orthofibration, and let $g \colon y \to y'$ and $h \colon z \to z'$ be morphisms in $Y$ and $Z$, respectively.
    Then the following are equivalent:
    \begin{enumerate}
        \item The cartesian transport $g^* \colon p_1^{-1}(a') \to p_1^{-1}(a)$ along $g$ preserves $p_2$-co\-car\-te\-sian lifts of $h$.
        \item The cocartesian transport $h_! \colon p_2^{-1}(b) \to p_2^{-1}(b')$ along $h$ preserves $p_1$-car\-te\-sian lifts of $g$.
    \end{enumerate}
\end{lem}
\begin{proof}
    This follows by inspection of the proof of \cite[Proposition~2.3.11]{hhln:lax-adjunctions}.
\end{proof}

Given \cref{lem:preserve-lifts}, we can make the following definition, see \cite[Definition~2.3.10]{hhln:lax-adjunctions}.

\begin{defn}
    A curved orthofibration 
    \[ p = (p_1,p_2) \colon X \to Y \times Z \]
    is an \emph{orthofibration} if it satisfies the following equivalent conditions:
    \begin{enumerate}
        \item the cartesian transport functor $g^* \colon p_1^{-1}(y') \to p_1^{-1}(y)$ preserves $p_2$-co\-car\-te\-sian morphisms for every morphism $g \colon y \to y'$ in $Y$;
        \item the cocartesian transport functor $h_! \colon p_2^{-1}(z) \to p_2^{-1}(x')$ preserves $p_1$-car\-te\-sian morphisms for every morphism $h \colon z \to z'$ in $Z$.
    \end{enumerate}
\end{defn}

By \cite[Corollary~2.5.6]{hhln:lax-adjunctions}, the equivalences of \cref{prop:straighten-curved-ortho} restrict to equivalences
\begin{equation}\label{eq:straighten-orthofibs}
    \Fun(Y^\op,\Cocart(Z)) \simeq \Ortho(Y,Z) \simeq \Fun(Z, \Cart(Y)),
\end{equation}
where $\Ortho(Y,Z) \subseteq \CurvOrtho(Y,Z)$ denotes the full subcategory of orthofibrations.
Consequently, every orthofibration straightens to a functor $Y^\op \times Z \to \catinf$, either by cocartesian straightening over $Z$ or by cartesian straightening over $Y$.
The resulting straightening functors $\Ortho(Y,Z) \to \Fun(Y^\op \times Z,\catinf)$ are equivalent, see \cite[Remark~2.5.7]{hhln:lax-adjunctions}.

Since $(s,t) \colon \Ar^\oplax \to \catinf \times \catinf$ is an orthofibration by \cite[Proposition~7.9]{hhln:two-variable-fibrations},
combining \cite[Theorem~E]{hhln:lax-adjunctions} and \cite[Theorem~7.21]{hhln:two-variable-fibrations} shows that $(s,t)$ straightens to the functor
\[ \Fun \colon \catinf^\op \times \catinf \to \catinf. \]
In fact, we will reprove this statement without recourse to the 2-categorical machinery of \cite{hhln:two-variable-fibrations}.

\begin{theorem}\label{thm:main}
    The functor
    \[ (s_*,t_*) \colon \Alg_\cO(\Ar^\oplax) \to \Alg_\cO(\catinf) \times \Alg_\cO(\catinf) \]
    is an orthofibration which straightens to the obvious functor
    \[ (\cX,\cY) \mapsto \Alg_{\cX/\cO}(\cY). \]
\end{theorem}

\begin{rem}\label{rem:aropl-straightened}
    By specialising \cref{thm:main} to the case that $\cO^\otimes$ is the trivial operad, we recover the claim that the orthofibration $(s,t) \colon \Ar^\oplax \to \catinf \times \catinf$ straightens to the functor $\Fun \colon \catinf^\op \times \catinf \to \catinf$.
\end{rem}

Assuming this statement, we can prove \cref{thm:day-improved}.

\begin{proof}[Proof of \cref{thm:day-improved}]
    We will apply \cref{thm:main} to $\cA^\otimes$, which provides for every pair of $\cA$-monoidal categories $(\cX,\cY)$ an equivalence
    \[ \Alg_{/\cA}(\Day_{\cX,\cY}) \simeq \Alg_{\cX/\cA}(\cY). \]
    Therefore, we obtain equivalences
    \begin{align*}
        \Alg_{\cA/\cO}(\Day_{\cC,\cD})
        &\simeq \Alg_{/\cA}(\cA \times_\cO \Day_{\cC,\cD}) \\
        &\simeq \Alg_{/\cA}(\Day_{\cA \times_\cO \cC,\cA \times_\cO \cD}) \\
        &\simeq \Alg_{\cA \times_\cO \cC/\cA}(\cA \times_\cO \cD) \\
        &\simeq \Alg_{\cA \times_\cO \cC/\cO}(\cD).
    \end{align*}
    Specialising to the case $\cA^\otimes = \Day_{\cC,\cD}^\otimes$, this provides an evaluation map
    \[ \Day_{\cC,\cD}^\otimes \times_{\cO^\otimes} \cC^\otimes \to \cD^\otimes \]
    over $\cO^\otimes$ which induces the above equivalence, making it clear that this identification is natural.
\end{proof}

Before embarking on the proof of \cref{thm:main}, let us record another easy consequence.
We require some additional notation.

\begin{defn}\label{def:lax-slices}
    Define the operads $\lslice{\catinf}{\cC}{}^\otimes$ and $\lslice{\catinf}{}{\cD}^\otimes$ by the pullbacks
    \[\begin{tikzcd}
        \lslice{\catinf}{\cC}{}^\otimes\ar[r]\ar[d] & (\Ar^\oplax)^\times\ar[d, "{(s,t)}"] \\
        \cO^\otimes \times_{\Comm^\otimes} \catinf^\times\ar[r, "\cC \times \id"] & \catinf^\times \times_{\Comm^\otimes} \catinf^\times
    \end{tikzcd}\]
    and
    \[\begin{tikzcd}
        \lslice{\catinf}{}{\cD}^\otimes\ar[r]\ar[d] & (\Ar^\oplax)^\times\ar[d, "{(s,t)}"] \\
         \catinf^\times \times_{\Comm^\otimes} \cO^\otimes\ar[r, "\id \times \cD"] & \catinf^\times \times_{\Comm^\otimes} \catinf^\times
    \end{tikzcd}\]
\end{defn}

\begin{cor}
    Let $\alpha \colon \cA^\otimes \to \cO^\otimes$ be an operad over $\cO^\otimes$.
    \begin{enumerate}
        \item The functor
        \[ t_\cC \colon \Alg_{\cA/\cO}(\lslice{\catinf}{\cC}{}) \to \Alg_{\cA/\cO}( \cO \times \catinf) \simeq \Alg_\cA(\catinf) \]
        is a cocartesian fibration which straightens to the functor
        \[ \Alg_\cA(\catinf) \to \catinf,\quad \cY \mapsto \Alg_{\cA \times_\cO \cC/\cA}(\cY) \]
        whose functoriality is given by postcomposition with $\cA$-monoidal functors.
        \item The functor
        \[ s_\cD \colon \Alg_{\cA/\cO}(\lslice{\catinf}{}{\cD}) \to \Alg_{\cA/\cO}(\catinf \times \cO) \simeq \Alg_\cA(\catinf) \]
        is a cartesian fibration which straightens to the functor
        \[ \Alg_\cA(\catinf)^\op \to \catinf,\quad \cX \mapsto \Alg_{\cX/\cO}(\cD) \]
        whose functoriality is given by precomposition with $\cA$-monoidal functors.
    \end{enumerate}
\end{cor}
\begin{proof}
    This follows immediately from \cref{thm:main} applied to the base operad $\cA^\otimes$ together with the naturality of unstraightening. 
\end{proof}

The remainder of this section is concerned with the proof of \cref{thm:main}.
The key input for our argument is the existence of free cartesian fibrations.

\begin{const}\label{const:free}
 Let $I$ be a small category and let $f \colon X \to I$ be a functor.
 Define the \emph{free cartesian fibration} $\Fr(f)$ on $f$ as the pullback
 \[\begin{tikzcd}
  \Fr_\cart(f)\ar[r]\ar[d] & \Fun([1],I)\ar[d, "\ev_1"] \\
  X\ar[r, "f"] & I
 \end{tikzcd}\]
 together with the evaluation map
 \[ \Fr_\cart(f) \to \Fun([1],I) \xrightarrow{\ev_0} I. \]
\end{const}

\begin{prop}[{\cite[Theorem~4.5]{ghn:free-fibrations}}]\label{prop:free-cart}
    Let $I$ be a category.
    The functor $\Cart(I) \to \catinf_{/I}$ admits a left adjoint
    \[ \Fr_\cart \colon \catinf_{/I} \to \Cart(I) \]
    which sends $f \colon X \to I$ to $\Fr_\cart(f)$.
\end{prop}

\begin{rem}
 Dualising \cref{prop:free-cart} shows that the pullback
 \[\begin{tikzcd}
  \Fr_\cocart(f)\ar[r]\ar[d] & \Fun([1],I)\ar[d, "\ev_0"] \\
  X\ar[r, "f"] & I
 \end{tikzcd}\]
 together with the evaluation map
 \[ \Fr_\cocart(f) \to \Fun([1],I) \xrightarrow{\ev_1} I \]
 is the free cocartesian fibration on $f$.
\end{rem}

\begin{prop}\label{prop:aropl-orthofib}
    \ \begin{enumerate}
     \item\label{prop:aropl-orthofib-1} The functor
        \[ (s,t) \colon \Cart([1]) \to \catinf \times \catinf \]
        is an orthofibration.
        A morphism $f \colon p \to q$ in $\Cart([1])$ is $s$-cartesian if and only if $t(f)$ is an equivalence, and $f$ is $t$-cocartesian if and only if $s(f)$ is an equivalence.
     \item\label{prop:aropl-orthofib-2} The functor
        \[ (s,t) \colon \Ar^\oplax \to \catinf \times \catinf \]
        is an orthofibration.
        Moreover, the functor $\Cart([1]) \to \Ar^\oplax$ preserves both $s$-cartesian and $t$-cocartesian morphisms.
    \end{enumerate}
\end{prop}
\begin{proof}
    By unstraightening, $\Cart([1]) \simeq \Fun([1]^\op,\catinf)$, with $s$ and $t$ corresponding to evaluation at $1$ and $0$, respectively.
    By \cite[Corollary~2.4.7.11]{HTT}, the evaluation functor at $0$ is a cartesian fibration and the evaluation functor at $1$ is a cocartesian fibration, and the characterisation of $s$-cartesian and $t$-cocartesian morphisms follows from \cite[Lemma~2.4.7.5]{HTT}.
    The explicit description of $s$-cartesian and $t$-cocartesian morphisms also implies that $(s,t)$ is an orthofibration.
    
    For assertion~\eqref{prop:aropl-orthofib-2}, it suffices to show that $s$-cartesian morphisms in $\Cart([1])$ are also $s$-cartesian in $\Ar^\oplax$, and that the same holds true for $t$-cocartesian morphisms.
    
    So let $f \colon p \to q$ be an $s$-cartesian morphism in $\Cart([1])$, where $p \colon X \to [1]$ and $q \colon Y \to [1]$ are cartesian fibrations.
    For every cartesian fibration $r \colon Z \to [1]$, we have to show that the commutative square
    \[\begin{tikzcd}
	  \Hom_{\catinf_{/[1]}}(r,p)\ar[r, "f \circ -"]\ar[d] & \Hom_{\catinf_{/[1]}}(r,q)\ar[d] \\
	  \Hom_{\catinf}(s(r),s(p))\ar[r, "s(f) \circ -"] & \Hom_{\catinf}(s(r),s(q))
	\end{tikzcd}\]
	is a pullback.
	Observing that $s(\Fr_\cart(r)) \simeq [1]_{1/} \times s(r) \simeq s(r)$, this square is identified via \cref{prop:free-cart} with the commutative square
	\[\begin{tikzcd}
	  \Hom_{\Cart([1])}(\Fr_\cart(r),p)\ar[r]\ar[d] & \Hom_{\Cart([1])}(\Fr_\cart(r),q)\ar[d] \\
	  \Hom_{\catinf}(s(r),s(p))\ar[r, "s(f) \circ -"] & \Hom_{\catinf}(s(r),s(q))
	\end{tikzcd}\]
	which is a pullback by assumption.
	
	Let now $f \colon p \to q$ be a $t$-cocartesian morphism in $\Cart([1])$.
	Let $r \colon Z \to [1]$ be an arbitrary cartesian fibration.
	Note that $t(\Fr_\cart(p)) \simeq X$ and $t(\Fr_\cart(q)) \simeq Y$.
	Using \cref{prop:free-cart} once more, it suffices to show that the outer square in the commutative diagram
	\begin{equation}\label{eq:aropl-orthofib}
	\begin{tikzcd}[column sep=3em]
	  \Hom_{\Cart([1])}(\Fr_\cart(q),r)\ar[r, "- \circ \Fr(f)"]\ar[d, "t"'] & \Hom_{\Cart([1])}(\Fr_\cart(p),r)\ar[d, "t"] \\
	  \Hom_{\catinf}(Y,t(r))\ar[r, "- \circ f"]\ar[d] & \Hom_{\catinf}(X,t(r))\ar[d] \\
	  \Hom_{\catinf}(t(q),t(r))\ar[r, "- \circ t(f)"] & \Hom_{\catinf}(t(p),t(r))
	\end{tikzcd}
	\end{equation}
	is a pullback.
    Since $s(\Fr_\cart(f)) \simeq s(f)$ is an equivalence, the morphism $\Fr_\cart(f)$ is $t$-cocartesian, which means that the top square is a pullback.
	For the lower square, we use the explicit formula for cartesian unstraightening over $[1]^\op$ from \cite[Proposition~3.1]{ghn:free-fibrations}: since the left square and outer square in the commutative diagram
	\[\begin{tikzcd}
	  s(p)\ar[r, "\St(p)"]\ar[d, "1 \times \id"'] & t(p)\ar[r, "t(f)"]\ar[d] & t(q)\ar[d] \\
	  {[1]}^\op \times s(p)\ar[r] & X\ar[r, "f"] & Y
	\end{tikzcd}\]
	are pushouts, so is the right square.
	This implies that the bottom square in \eqref{eq:aropl-orthofib} is a pullback.
	Consequently, the outer square in \eqref{eq:aropl-orthofib} is a pullback as required.
\end{proof}

Since we are considering the cartesian symmetric monoidal structures on $\Ar^\oplax$ and $\catinf$, we can bootstrap the analogous statements for categories of $\cO$-algebras from this.

\begin{lem}\label{lem:orthofib-fun}
    Let $p = (p_1,p_2) \colon X \to Y \times Z$ be a curved orthofibration/an orthofibration and let $I$ be a small category.
    Then
    \[ p_* = ((p_1)_*,(p_2)_*) \colon \Fun(I,X) \to \Fun(I,Y) \times \Fun(I,Z) \]
    is also a curved orthofibration/an orthofibration.
    The relevant cartesian and cocartesian morphisms are given by those natural transformations whose components are all cartesian or cocartesian, respectively.
\end{lem}
\begin{proof}
    This is immediate from \cite[Corollary~3.2.2.12]{HTT}.
\end{proof}

\begin{lem}\label{lem:monoid-orthofib}
 Let $X$, $Y$ and $Z$ be categories with finite products.
 Suppose that $p = (p_1,p_2) \colon X \to Y \times Z$ is a functor such that
 \begin{enumerate}
  \item\label{lem:monoid-orthofib-1} $p$ is a curved orthofibration;
  \item\label{lem:monoid-orthofib-2} $p$ preserves finite products;
  \item\label{lem:monoid-orthofib-3} finite products of $p_1$-cartesian morphisms are $p_1$-cartesian;
  \item\label{lem:monoid-orthofib-4} finite products of $p_2$-cocartesian morphisms are $p_2$-cocartesian.
 \end{enumerate}
 Then the induced functor $p_* \colon \Mon_\cO(X) \to \Mon_\cO(Y) \times \Mon_\cO(Z)$ is a curved orthofibration.
 A morphism in $\Mon_\cO(X)$ is $(p_1)_*$-cartesian or $(p_2)_*$-cocartesian if and only if it is $(p_1)_*$-cartesian or $(p_2)_*$-cocartesian in $\Fun(\cO^\otimes,X)$.
\end{lem}
\begin{proof}
  By \cref{lem:orthofib-fun}, the induced functor
  \[ p_* = ((p_1)_*,(p_2)_*) \colon \Fun(\cO^\otimes,X) \to \Fun(\cO^\otimes,Y) \times \Fun(\cO^\otimes,Z) \]
  also satisfies properties~\eqref{lem:monoid-orthofib-1} -- \eqref{lem:monoid-orthofib-4}.
  
  Let $g \colon N \to N'$ be a morphism in $\Mon_\cO(Y)$ and let $M'$ be an $\cO$-monoid in $X$ lifting $N'$.
  Considering $g$ as a morphism in $\Fun(\cO^\otimes,Y)$, there exists a cartesian lift $f \colon M \to M'$ in $\Fun(\cO^\otimes,X)$.
  We claim that $M$ is also an $\cO$-monoid in $X$.
  For $x = x_1 \boxtimes \ldots \boxtimes x_n$ in $\cO^\otimes_{\langle n \rangle}$, the Segal maps of $M$ and $M'$ fit into a commutative square
  \[\begin{tikzcd}
   M(x)\ar[r]\ar[d] & \prod_{i=1}^n M(x_i)\ar[d] \\
   M'(x)\ar[r, "\simeq"] & \prod_{i=1}^n M'(x_i)
  \end{tikzcd}\]
  Since products of cartesian morphisms in $X$ are cartesian, both vertical arrows are cartesian morphisms.
  By \cite[Lemma~2.4.2.7]{HTT}, it follows that the top horizontal morphism is also cartesian.
  Since $N$ is an $\cO$-monoid, this morphism is a cartesian lift of an equivalence, and therefore itself an equivalence.
  It follows that $(p_1)_* \colon \Mon_\cO(X) \to \Mon_\cO(Y)$ is a cartesian fibration.
  
  Since $(p_1)_*$-cartesian lifts in $\Fun(\cO^\otimes,X)$ are characterised by being pointwise $p_1$-cartesian, it also follows that $(p_1)_*$-cartesian morphisms project to equivalences under $(p_2)_*$.
  
  The argument for $(p_2)_*$ is completely analogous.
\end{proof}

In particular, we obtain the following.

\begin{prop}\label{prop:alg-s-t-orthofibration}
    The functor
    \[ (s_*,t_*) \colon \Alg_\cO(\Ar^\oplax) \to \Alg_\cO(\catinf) \times \Alg_\cO(\catinf) \]
    is an orthofibration.
    Both $s_*$-cartesian and $t_*$-cocartesian morphisms are detected by the functor
    \[ \Alg_\cO(\Ar^\oplax) \xrightarrow{\sim} \Mon_\cO(\Ar^\oplax) \to \Fun(\cO^\otimes,\Ar^\oplax). \]
\end{prop}
\begin{proof}
    Due to \cref{prop:aropl-orthofib}, the functor $(s,t) \colon \Ar^\oplax \to \catinf \times \catinf$ is an orthofibration satisfying the assumptions of \cref{lem:monoid-orthofib}.
    The proposition follows.
\end{proof}

To determine the straightening of $(s_*,t_*)$, we require some additional preparation.
First, we observe that the existence of free (co)cartesian fibrations implies the existence of free orthofibrations.

\begin{cor}\label{prop:free-orthofib}
    The inclusion functor $\Ortho(Y,Z) \to \catinf_{/Y \times Z}$ admits a left adjoint
    \[ \Fr \colon \catinf_{/Y \times Z} \to \Ortho(Y,Z). \]
\end{cor}
\begin{proof}
    We write $p_Z \colon Y \times Z \to Z$ for the projection functor.
    Since $p_Z$ is a cartesian fibration, the equivalence $\catinf_{/Y \times Z} \simeq (\catinf_{/Y})/p_Z$ and \cref{prop:free-cart} induce an adjunction
    \[ \catinf_{/Y \times Z} \rightleftarrows \Cart(Y)/p_Z. \]
    After identifying
    \[ \Cart(Y)/p_Z \simeq \Fun(Y^\op,\catinf)_{/\con_Z} \simeq \Fun(Y^\op,\cat_{/Z}), \]
    the existence of free cocartesian fibrations induces an adjunction
    \[ \Cart(Y)/p_Z \rightleftarrows \Fun(Y^\op,\Cocart(Z)) \simeq \Ortho(Y,Z).\qedhere \]
\end{proof}

\begin{rem}
    Unwinding the proof of \cref{prop:free-orthofib}, one finds that the free orthofibration on a functor $f \colon X \to Y \times Z$ is given by the pullback
    \[\begin{tikzcd}
        \Fr(f)\ar[r]\ar[d] & \Fun([1],Y) \times \Fun([1],Z)\ar[d, "\ev_1 \times \ev_0"] \\
        X\ar[r, "f"] & Y \times Z
    \end{tikzcd}\]
    together with the evaluation map
    \[ \Fr(f) \to \Fun([1],Y) \times \Fun([1],Z) \xrightarrow{\ev_0 \times \ev_1} Y \times Z. \]
    One can adapt the proof of \cite[Theorem~4.5]{ghn:free-fibrations} to show directly that this yields a left adjoint to the functor $\Ortho(Y,Z) \to \catinf_{/Y \times Z}$.
\end{rem}

We can apply this statement to identify the straightenings of cotensors of orthofibrations.

\begin{lem}\label{lem:orthofib-cotensor}
    Let $p \colon X \to Y \times Z$ be an orthofibration.
    Consider the pullback
    \[\begin{tikzcd}
        X^I\ar[r]\ar[d, "p^I"'] & \Fun(I,X)\ar[d, "p_*"] \\
        Y \times Z\ar[r, "\con"] & \Fun(I,Y \times Z)
    \end{tikzcd}\]
    Then $p^I = (p^I_1,p^I_2)$ is an orthofibration such that both $p^I_1$-cartesian and $p^I_2$-cocartesian morphisms are detected componentwise in $\Fun(I,X)$.
    Moreover, $p^I$ straightens to the functor
    \[ \Fun(I,\St(p)) \colon Y^\op \times Z \to \catinf. \]
\end{lem}
\begin{proof}
    The first part of the lemma is precisely \cref{lem:orthofib-fun}, so we only have to prove the assertion about the straightening of $p^I$.
    For every functor $f \colon T \to Y \times Z$, there exist by \cref{prop:free-orthofib} natural equivalences
    \begin{align*}
        \Hom_{Y \times Z}(T,X^I)
        &\simeq \Hom_{Y \times Z}(T \times I,X) \\
        &\simeq \Hom_{\Ortho(Y,Z)}(\Fr(T \times I \to Y \times Z),X) \\
        &\simeq \Hom_{\Ortho(Y,Z)}(\Fr(f) \times I,X) \\
        &\simeq \Nat(\St(\Fr(f)) \times I,\St(p)) \\
        &\simeq \Nat(\St(\Fr(f)), \Fun(I,\St(p))) \\
        &\simeq \Hom_{\Ortho(Y,Z)}(\Fr(f),\Un(\Fun(I,\St(p)))) \\
        &\simeq \Hom_{Y \times Z}(T,\Un(\Fun(I,\St(p)))),
    \end{align*}
    which implies the lemma.
\end{proof}

Recall the following definition from \cite[Proposition~2.3.13]{hhln:lax-adjunctions}.

\begin{defn}
    A functor $p = (p_1,p_2) \colon X \to Y \times Z$ is a \emph{bifibration} if the following conditions are satisfied:
    \begin{enumerate}
        \item $p_1$ is a cartesian fibration such that a morphism in $X$ is $p_1$-cartesian if and only if it projects to an equivalence under $p_2$;
        \item $p_2$ is a cocartesian fibration such that a morphism in $X$ is $p_2$-cocartesian if and only if it projects to an equivalence under $p_1$.
    \end{enumerate}
\end{defn}

By \cite[Corollary~2.3.15]{hhln:lax-adjunctions}, the equivalences from~\eqref{eq:straighten-orthofibs} restrict to equivalences
\[ \Fun(Y^\op, \LFib(Z)) \simeq \Bifib(Y,Z) \simeq \Fun(Z,\RFib(Y)), \]
where $\Bifib(Y,Z) \subseteq \Ortho(Y,Z)$ denotes the full subcategory of bifibrations, and $\LFib(Z)$ and $\RFib(Y)$ denote the categories of left fibrations over $Z$ and right fibrations over $Y$, respectively.

\begin{ex}\label{ex:arrow-cat}
    The functor $(\ev_1,\ev_0) \colon \Fun([1]^\op,X) \to X$ is a bifibration for every category $X$---the special case $X=\catinf$ was covered in \cref{prop:aropl-orthofib}.
    Moreover, this functor straightens to the functor
    \[ \Hom_X \colon X^\op \times X \to \Spc. \]
    This follows for example from \cite[Corollary~A.2.5]{hms:shifted} because the cartesian unstraightening of $\Hom_X$ is the twisted arrow category.
\end{ex}

\begin{defn}
    Let $p = (p_1,p_2) \colon X \to Y \times Z$ be an orthofibration.
    Define $X_\bicart$ as the wide subcategory of $X$ generated by the collections of $p_1$-cartesian and $p_2$-cocartesian morphisms.
\end{defn}

\begin{lem}\label{lem:underlying-bifib}
    Let $p = (p_1,p_2) \colon X \to Y \times Z$ be an orthofibration.
    \begin{enumerate}
        \item\label{lem:underlying-bifib-1} The following are equivalent for a morphism $f$ in $X$:
         \begin{enumerate}
             \item $f$ lies in $X_\bicart$;
             \item $f$ is the composition of a $p_1$-cartesian morphism followed by a $p_2$-cocartesian morphism;
             \item $f$ is the composition of a $p_2$-cocartesian morphism followed by a $p_1$-cartesian morphism.
         \end{enumerate}
        \item\label{lem:underlying-bifib-2} The restriction $p_\bicart \colon X_\bicart \to Y \times Z$ of $p$ is a bifibration. The inclusion functors $X_\bicart \to X$ assemble to the counit transformation of an adjunction
        \[ \inc \colon \Bifib(Y,Z) \rightleftarrows \Ortho(Y,Z) \cocolon (-)_\bicart. \]
        \item\label{lem:underlying-bifib-3} The bifibration $p_\bicart$ straightens to the functor
        \[ Y^\op \times Z \xrightarrow{\St(p)} \catinf \xrightarrow{\iota} \Spc. \]
    \end{enumerate}
\end{lem}
\begin{proof}
    By \cite[Definition~2.3.10]{hhln:lax-adjunctions}, $p_1$-cartesian morphisms canonically commute with $p_2$-cocartesian morphisms, which shows assertion~\eqref{lem:underlying-bifib-1}.

    For assertion~\eqref{lem:underlying-bifib-2}, let us first check that $(p_\bicart)_1 \colon X_\bicart \to Y$ is a cartesian fibration.
    Every morphism in $Y$ admits a $p_1$-cartesian lift, so this reduces to checking that for a $p_1$-cartesian morphism $\xi \colon x \to x'$ in $X$, an arbitrary morphism $\alpha \colon a \to x$ lies in $X_\bicart$ if and only if $\xi \circ \alpha$ lies in $X_\bicart$.
    Writing $\alpha$ as the composite of a $p_2$-cocartesian morphism followed by a $p_1$-cartesian morphism, this is immediate from \cite[Lemma~2.4.2.7]{HTT}.
    In particular, every $(p_\bicart)_1$-cartesian morphism projects to an equivalence under $(p_\bicart)_2$.

    Suppose now that $\xi \colon x \to x'$ is a morphism in $X_\bicart$ such that $p_2(\xi)$ is an equivalence.
    Writing $\xi = \xi_\cocart \circ \xi_\cart$ as a composition of a $p_1$-cartesian morphism followed by a $p_2$-cocartesian morphism, it follows that $p_2(\xi_\cocart)$ is an equivalence.
    Hence $\xi_\cocart$ is a $p_2$-cocartesian lift of an equivalence, and thus an equivalence.
    It follows that $\xi$ is $p_1$-cartesian, and therefore also $(p_\bicart)_1$-cartesian.

    By dualising, we see that $(p_\bicart)_2 \colon X \to Z$ is a cocartesian fibration such that a morphism is $(p_\bicart)_2$-cocartesian if and only if it projects to an equivalence under $(p_\bicart)_1$.
    Hence $p_\bicart$ is a bifibration.

    Since morphisms in $\Ortho(Y,Z)$ preserve all relevant cartesian and cocartesian morphisms, it is immediate that the inclusion $X_\bicart \to X$ induces an equivalence
    \[ \Hom_{\Bifib(Y,Z)}(T,X_\bicart) \xrightarrow{\sim} \Hom_{\Ortho(Y,Z)}(T,X) \]
    for every bifibration $T \to Y \times Z$.
    
    Assertion~\eqref{lem:underlying-bifib-3} follows from the commutative diagram
    \[\begin{tikzcd}
        \Bifib(Y,Z)\ar[r, "\sim"]\ar[d, "\inc"] & \Fun(Z, \RFib(Y))\ar[r, "\sim"]\ar[d, "\inc"] & \Fun(Z, \Fun(Y^\op,\Spc))\ar[d, "\inc"] \\
        \Ortho(Y,Z)\ar[r, "\sim"] & \Fun(Z, \Cart(Y))\ar[r, "\sim"] & \Fun(Z,\Fun(Y^\op,\catinf))
    \end{tikzcd}\]
    by passing to right adjoints.
\end{proof}

Finally, recall that slice categories of $\Op$ are cotensored over $\catinf$ as follows.
For $I$ a small category and $\phi \colon \cX^\otimes \to \cB^\otimes$ an operad map, define $\Fun(I,\cX)^\otimes$ as the pullback
\[\begin{tikzcd}
    \Fun(I,\cX)^\otimes\ar[r]\ar[d] & \Fun(I,\cX^\otimes)\ar[d, "\phi_*"] \\
    \cB^\otimes\ar[r, "\con"] & \Fun(I,\cB^\otimes)
\end{tikzcd}\]
This operad has the universal property that
\[ \Alg_{\cA/\cB}(\Fun(I,\cX)) \simeq \Fun(I, \Alg_{\cA/\cB}(\cX)) \]
for every operad $\cA^\otimes \to \cB^\otimes$ over $\cB^\otimes$ \cite[Remark~2.1.3.4]{HA}.

If $\pi \colon \cX^\otimes \to \cX$ is a cartesian structure, one checks directly that
\[ \Fun(I,\cX)^\otimes \to \Fun(I,\cX^\otimes) \xrightarrow{\pi \circ - } \Fun(I,\cX) \]
exhibits $\Fun(I,\cX)^\otimes$ as a cartesian structure on $\Fun(I,\cX)$.
In particular, one obtains the cartesian symmetric monoidal structure
\[ \Fun(I,\Ar^\oplax)^\times \to \Comm^\otimes \]
by applying this construction to the operad $(\Ar^\oplax)^\times \to \Comm^\otimes$.

\subsection*{Proof of \texorpdfstring{\cref{thm:main}}{the main theorem}}
We can now finish the proof of our main result.
By \cref{prop:alg-s-t-orthofibration}, the functor
\[ (s_*,t_*) \colon \Alg_\cO(\Ar^\oplax) \to \Alg_\cO(\catinf) \times \Alg_\cO(\catinf) \]
is an orthofibration.
In combination with \cref{prop:aropl-orthofib}, we obtain a characterisation of the $s_*$-cartesian morphisms as those morphisms which map under
\[ \Alg_\cO(\Ar^\oplax) \to \Fun(\cO^\otimes,\Ar^\oplax) \]
to a natural transformation whose components all preserve cartesian morphisms and project to an equivalence under $t \colon \Ar^\oplax \to \catinf$.
Analogously for $t_*$-cocartesian morphisms.
    
We will identify the composite of the straightening of $(s_*,t_*)$ with the Yoneda embedding $\yo \colon \catinf \to \cP(\catinf)$.
By \cref{lem:orthofib-cotensor}, the composite
\begin{align*}
    \catinf^\op
    &\xrightarrow{(s_*,t_*)^{(-)}} \Ortho(\Alg_\cO(\catinf),\Alg_\cO(\catinf)) \\
    &\xrightarrow{\St} \Fun(\Alg_\cO(\catinf)^\op \times \Alg_\cO(\catinf),\catinf) \\
    &\xrightarrow{\iota \circ -} \Fun(\Alg_\cO(\catinf)^\op,\Spc)
\end{align*}
corresponds to $\yo \circ \St(s_*,t_*)$ after currying.
By virtue of \cref{lem:underlying-bifib}, the composite $(\iota \circ -) \circ \St$ is equivalent to the functor that first applies $(-)_\bicart$ and then straightens the resulting bifibration.
This leaves us with identifying the bifibrations $(s_*,t_*)^I_\bicart$.

Consider the natural fully faithful functor
\begin{align*}
    \Psi' \colon \Fun(I,\Alg_\cO(\Ar^\oplax))
    &\simeq \Alg_\cO(\Fun(I,\Ar^\oplax)) \\
    &\xrightarrow{\sim} \Mon_\cO(\Fun(I,\Ar^\oplax)) \\
    &\subseteq \Fun(\cO^\otimes, \Fun(I,\Ar^\oplax)) \simeq \Fun(\cO^\otimes \times I, \Ar^\oplax).
\end{align*}
Then $\Fun(I,s_*)$-cartesian morphisms in the domain correspond precisely to those natural transformations $\tau$ in the target with the property that for all $x \in \cO^\otimes$ and $i \in I$, the functor $\tau(x,i)$ preserves cartesian morphisms over $[1]$ and $t(\tau(x,i))$ is an equivalence.
From the analogous assertion for $\Fun(I,t_*)$-cocartesian morphisms, it follows that $\Psi'$ restricts to a fully faithful functor
\[ \Fun(I,\Alg_\cO(\Ar^\oplax))_\bicart \to \Fun(\cO^\otimes \times I, \Ar^\oplax)^\cart. \]
Composing with the natural equivalence of \cref{prop:straighten-curved-ortho}, we obtain a fully faithful functor
\[ \Psi \colon \Fun(I,\Alg_\cO(\Ar^\oplax))_\bicart \to \Fun([1]^\op, \Cocart^\lax(\cO^\otimes \times I))^\cocart. \]

\begin{lem}\label{lem:identify-monoids}
    The essential image of $\Psi$ comprises of those functors
    \[ E \colon [1]^\op \to \Cocart^\lax(\cO^\otimes \times I) \]
    satisfying the following conditions:
    \begin{enumerate}
        \item for $k=0,1$ and $i \in I$, the functor $E(k)_i \to \cO^\otimes \times \{i\}$ is a cocartesian fibration of operads;
        \item the functor $E(1)_i \to E(0)_i$ preserves inert morphisms for every $i \in I$.
    \end{enumerate}
\end{lem}
\begin{proof}
    A functor $M \colon \cO^\otimes \to \Fun(I,\Ar^\oplax)$ is an $\cO$-monoid if and only if the evaluation $M_i \colon \cO^\otimes \to \Ar^\oplax$ is an $\cO$-monoid for all $i \in I$.
    Recall that $M_i$ being an $\cO$-monoid means that for every $x = x_1 \boxtimes \ldots \boxtimes x_n \in \cO^\otimes_{\langle n \rangle}$, the appropriate inert morphisms induce equivalences
    \[\rho \colon  M_i(x) \xrightarrow{\sim} M_i(x_1) \mathop{\times}\limits_{[1]} \ldots \mathop{\times}\limits_{[1]} M_i(x_n) \]
    of cartesian fibrations over $[1]$.
    This is the case if and only if each $\rho$ preserves cartesian morphisms and induces fibrewise equivalences.

    Denote by $E \colon [1]^\op \to \Cocart(\cO^\otimes \times I)$ the image of $M$ under $\Psi$, and let $E(k)_i$ be the restriction of $E(k)$ to $\cO^\otimes \times \{i\}$.
    The cocartesian fibration $E(k)_i \to \cO^\otimes$ is given by the unstraightening of the composite $\cO^\otimes \to \Fun(I,\Ar^\oplax) \xrightarrow{ev_i} \Ar^\oplax \xrightarrow{(-)|_{\{k\}}} \catinf$, which is an $\cO$-monoid because $\ev_i$ preserves products.
    Consequently, each map $\rho$ is a fibrewise equivalence if and only if both $E(0)_i$ and $E(1)_i$ are cocartesian fibrations of operads \cite[Example~2.4.2.4]{HA}.
    Since inert morphisms in $E(k)_i$ are precisely the cocartesian lifts of inert morphisms in $\cO^\otimes \times \{i\}$, \cref{lem:preserve-lifts} shows that $\rho$ preserves all cartesian morphisms if and only if $E(1)_i \to E(0)_i$ preserves all inert morphisms.
\end{proof}

Note that $\Psi$ fits into a natural commutative diagram
\[\begin{tikzcd}
    \Fun(I,\Alg_\cO(\Ar^\oplax))_\cart\ar[r, "\Psi"]\ar[d, "{(s,t)}"'] & \Fun([1]^\op, \Cocart^\lax(\cO^\otimes \times I))^\cocart\ar[d, "{(\ev_1,\ev_0)}"] \\
    \Fun(I,\Alg_\cO(\catinf)) \times \Fun(I,\Alg_\cO(\catinf))\ar[r, "\Phi"] & \Cocart(\cO^\otimes \times I) \times \Cocart(\cO^\otimes \times I)
\end{tikzcd}\]
with both $\Phi$ and $\Psi$ fully faithful.
The essential image of $\Phi$ comprises precisely of those pairs of functors whose restriction to $\cO^\otimes \times \{i\}$ is a cocartesian fibration of operads for every $i \in I$.
In particular, this induces a natural fully faithful functor $\Psi_I$ from $\Alg_\cO(\Ar^\oplax)^I_\bicart$ to the pullback of
\[\begin{tikzcd}[column sep=5em]
    & \Fun([1]^\op, \Cocart^\lax(\cO^\otimes \times I))^\cocart\ar[d, "{(\ev_1,\ev_0) =: \epsilon}"] \\
    \Cocart(\cO^\otimes) \times \Cocart(\cO^\otimes)\ar[r, "(- \times I) \times (- \times I)"] & \Cocart(\cO^\otimes \times I) \times \Cocart(\cO^\otimes \times I)
\end{tikzcd}\]
The right vertical evaluation functor is the pullback of
\[ (\ev_1,\ev_0) \colon \Fun([1]^\op, \Cocart^\lax(\cO^\otimes \times I)) \to \Cocart^\lax(\cO^\otimes \times I) \times \Cocart^\lax(\cO^\otimes \times I), \]
along the inclusion functor
\[ \Cocart(\cO^\otimes \times I) \times \Cocart(\cO^\otimes \times I) \to \Cocart^\lax(\cO^\otimes \times I) \times \Cocart^\lax(\cO^\otimes \times I). \]
Consequently, the naturality of unstraightening together with \cref{ex:arrow-cat} implies that this pullback of $\epsilon$ straightens to the functor
\[ \Hom_{\Cocart^\lax(\cO^\otimes \times I)}(- \times I, - \times I) \colon \Cocart(\cO^\otimes) \times \Cocart(\cO^\otimes) \to \Spc. \]
Observe in addition that
\[ \Hom_{\Cocart^\lax(\cO^\otimes \times I)}(- \times I, - \times I) \simeq \Hom_{\Cocart^\lax(\cO^\otimes)}(- \times I,-). \]
It follows from \cref{lem:identify-monoids} that $\Psi_I$ identifies the straightening of $\Alg_\cO(\Ar^\oplax)^I_\bicart$ with the full subfunctor of $\Hom_{\Cocart^\lax(\cO^\otimes)}(- \times I,-)$ given by those functors $\cX^\otimes \times I \to \cY^\otimes$ such that $\cX^\otimes \times \{i\} \to \cY^\otimes$ is an operad map for all $i \in I$.
After currying, we obtain a natural equivalence
\[ \St((s_*,t_*)^I_\bicart) \simeq \Hom_\catinf(I, \Alg_{\cX/\cO}(\cY)), \]
which is precisely what we needed to show.
\cref{thm:main} is now proved.


\section{Variation: Day convolution in suboperads of \texorpdfstring{$\catinf^\times$}{the cartesian symmetric monoidal structure on Cat}}\label{sec:variations}
From the preceding results, one can deduce analogous assertions for certain symmetric monoidal categories which arise as suboperads of $\catinf^\times$.
Consider a subcategory $\cU$ of $\catinf$ which is closed under finite products.
If $U_1$, $U_2$ and $T$ are objects in $\cU$, call a functor $F \colon U_1 \times U_2 \to T$ \emph{$\cU$-biexact} if both $F(u_1,-) \colon U_2 \to T$ and $F(-,u_2) \colon U_1 \to T$ are morphisms in $\cU$ for all $u_1 \in U_1$ and $u_2 \in U_2$.
There is an evident notion of a \emph{$\cU$-multiexact} functor for functors in more than two variables.

Assume that
\begin{enumerate}
    \item for each pair $U_1$ and $U_2$ of objects in $\cU$, there exists an initial $\cU$-biexact functor $U_1 \times U_2 \to U_1 \otimes U_2$;
    \item there exists a category $U \in \cU$ and an object $u \in U$ such that evaluation at $u$ induces an equivalence $\Hom_\cU(U,T) \xrightarrow{\sim} \iota T$.
\end{enumerate}
One example of a subcategory satisfying these conditions is the category $\catst$ of stable categories and exact functors.

Under these assumptions, $\cU$ refines to a symmetric monoidal category by considering the suboperad $\cU^\otimes$ of $\catinf^\times$ determined by the following conditions:
\begin{enumerate}
    \item the underlying category of $\cU^\otimes$ is $\cU$;
    \item morphisms $U_1 \boxtimes \ldots \boxtimes U_n \to T$ over the active map $\langle n \rangle \to \langle 1 \rangle$ correspond to $\cU$-multiexact functors $U_1 \times \ldots \times U_n \to T$.
\end{enumerate}
Observe that $\cO$-algebras in $\cU^\otimes$ correspond under unstraightening to cocartesian fibrations of operads over $\cO^\otimes$ whose fibres lie in $\cU$ and whose cocartesian transport functors are $\cU$-multiexact.

Consider now the suboperad $(\Ar^\oplax_\cU)^\otimes$ of $(\Ar^\oplax)^\times$ determined by the following properties:
\begin{enumerate}
    \item objects in the underlying category $\Ar^\oplax_\cU$ are given by cartesian fibrations $X \to [1]$ which straighten to functors $[1]^\op \to \cU$;
    \item morphisms are precisely those morphisms in $(\Ar^\oplax)^\times$ which map to the suboperad $\cU^\otimes \times_{\Comm^\otimes} \cU^\otimes$ under $(s,t)$.
\end{enumerate}
As before, given $\cO$-algebras $\cC$ and $\cD$ in $\cU$, we let $\lslice{\cU}{\cC}{}^\otimes$ and $\lslice{\cU}{}{\cD}^\otimes$ be given by the following pullbacks:
\[\begin{tikzcd}
    \lslice{\cU}{\cC}{}^\otimes\ar[r]\ar[d] & (\Ar^\oplax_\cU)^\otimes\ar[d, "{(s,t)}"] \\
    \cO^\otimes \times_{\Comm^\otimes} \cU^\otimes\ar[r, "\cC \times \id"] & \cU^\otimes \times_{\Comm^\otimes} \cU^\otimes
\end{tikzcd}
\quad
\begin{tikzcd}
    \lslice{\cU}{}{\cD}^\otimes\ar[r]\ar[d] & (\Ar^\oplax_\cU)^\otimes\ar[d, "{(s,t)}"] \\
    \cU^\otimes \times_{\Comm^\otimes} \cO^\otimes\ar[r, "\id \times \cD"] & \cU^\otimes \times_{\Comm^\otimes} \cU^\otimes
\end{tikzcd}\]
As a final piece of notation, denote by $\Alg_{\cC/\cO}^\cU(\cD)$ the full subcategory of $\Alg_{\cC/\cO}(\cD)$ spanned by those operad maps $\cC^\otimes \to \cD^\otimes$ over $\cO^\otimes$ such that $\cC^\otimes_x \to \cD^\otimes_x$ is a morphism in $\cU$ for every $x \in \cO^\otimes$.

\begin{prop}\label{prop:alg-aropl-s-straighten}
    Let $p \colon \cC^\otimes \to \cO^\otimes$ and $q \colon \cD^\otimes \to \cO^\otimes$ be cocartesian fibrations of operads corresponding to $\cO$-algebras in $\cU$. 
    \begin{enumerate}
        \item\label{prop:alg-aropl-s-straighten-1} The functor
        \[ (s_*,t_*) \colon \Alg_{\cO}(\Ar^\oplax_\cU) \to \Alg_\cO(\cU) \times \Alg_\cO(\cU) \]
        is an orthofibration which straightens to the functor
        \[ (\cX,\cY) \mapsto \Alg_{\cX/\cO}^\cU(\cY).\]
        \item\label{prop:alg-aropl-s-straighten-2} For every operad $\cA^\otimes \to \cO^\otimes$ over $\cO^\otimes$, the functor
        \[ t_\cC \colon \Alg_{\cA/\cO}(\lslice{\cU}{\cC}{}) \to \Alg_\cA(\cU) \]
        is a cocartesian fibration which straightens to the functor
        \[ \cY \mapsto \Alg_{\cA \times_\cO \cC/\cA}^\cU(\cY). \]
        \item\label{prop:alg-aropl-s-straighten-3} For every operad $\cA^\otimes \to \cO^\otimes$ over $\cO^\otimes$, the functor
        \[ s_\cD \colon \Alg_{\cA/\cO}(\lslice{\cU}{}{\cD}) \to \Alg_\cA(\cU) \]
        is a cartesian fibration which straightens to the functor
        \[ \cX \mapsto \Alg_{\cX/\cO}^\cU(\cD). \]
    \end{enumerate}
\end{prop}
\begin{proof}
    The operad $(\Ar^\oplax_\cU)^\otimes$ can be constructed in two steps.
    Consider first the pullback
    \[\begin{tikzcd}
        (\widetilde{\Ar^\oplax_\cU})^\otimes\ar[r]\ar[d] & (\Ar^\oplax)^\times\ar[d, "{(s,t)}"] \\
        \cU^\otimes \times_{\Comm^\otimes} \cU^\otimes\ar[r] & \catinf^\times \times_{\Comm^\otimes} \catinf^\times
    \end{tikzcd}\]
    Then $(\Ar^\oplax_\cU)^\otimes$ is the full subcategory of $(\widetilde{\Ar^\oplax_\cU})^\otimes$ spanned by objects corresponding to tuples of cartesian fibrations over $[1]$, each of which straightens to a functor $[1]^\op \to \cU$.
    Consequently, \cref{thm:main} implies that
    \[ (\tilde{s}_*,\tilde{t}_*) \colon \Alg_{\cO}\Big(\widetilde{\Ar^\oplax_\cU}\Big) \to \Alg_\cO(\cU) \times \Alg_\cO(\cU) \]
    is an orthofibration which straightens to the functor
    \[ (\cX,\cY) \mapsto \Alg_{\cX/\cO}(\cD). \]
    By construction, the fibre of $(\Ar^\oplax_\cU)^\otimes$ over $(\cX,\cY)$ is precisely the full subcategory $\Alg_{\cX/\cA}^\cU(\cY)$, and both the $\tilde{s}_*$-cartesian and $\tilde{t}_*$-cocartesian transport functors along morphisms in $\Alg_\cO(\cU)$ preserve these full subcategories.
    This identifies $\St(s_*,t_*)$ as the correct subfunctor.

    Assertions~\eqref{prop:alg-aropl-s-straighten-2} and \eqref{prop:alg-aropl-s-straighten-3} follow as before from \eqref{prop:alg-aropl-s-straighten-1}.
\end{proof}


\section{\texorpdfstring{Day convolution as an $\cO$-monoidal category}{Day convolution as an O-monoidal category}}\label{sec:day-o-monoidal}

Fix a base operad $\cO^\otimes$ as well as $\cO$-monoidal categories $\cC$ and $\cD$.
In this section, we reprove a well-known statement, see \cite[Proposition~2.2.6.16]{HA}, which is key for working with the Day convolution operad.
We include a proof to demonstrate that $\Day_{\cC,\cD}^\otimes$ is a feasible description of Day convolution.

For every operation $\phi \in \mathrm{Mul}_\cO(\{x_i\}_i,y)$, denote the associated tensor functors by $\otimes^\cC_\phi \colon \prod_i \cC(x_i) \to \cC(y)$ and $\otimes^\cD_\phi \colon \prod_i \cD(x_i) \to \cD(y)$.

\begin{prop}\label{prop:day-o-monoidal}
    For each $y \in \cO$, consider the following collection of slice categories:
    \[ \cK(y) := \left\{ \otimes^\cC_\psi/c \mid \psi \in \mathrm{Mul}_\cO(\{x_i\}_i,y),\ c \in \cC(y) \right\} \]
    Assume the following is true:
    \begin{enumerate}
        \item\label{prop:day-o-monoidal-1} for all $y \in \cO$, the category $\cD(y)$ admits all $\cK(y)$-shaped colimits;
        \item\label{prop:day-o-monoidal-2} for every operation $\phi \in \mathrm{Mul}_\cO(\{x_i\}_i,y)$ and every $j$, the associated tensor functor $\otimes^\cD_\phi \colon \prod_i \cD(x_i) \to \cD(y)$ preserves all $\cK(x_j)$-shaped colimits in the $j$-th component.
    \end{enumerate}
    Then $\Day_{\cC,\cD}^\otimes \to \cO^\otimes$ is a cocartesian fibration of operads.
\end{prop}

\begin{rem}
    The assumptions of \cref{prop:day-o-monoidal} are for example satisfied if there exists some regular cardinal $\kappa$ with the property that each $\cD(y)$ is $\kappa$-cocomplete, every tensor functor of $\cD$ preserves $\kappa$-small colimits in each variable, and each $\cC(y)$ is $\kappa$-small, reproducing \cite[Proposition~2.2.6.16]{HA}.
\end{rem}

\begin{proof}[Proof of \cref{prop:day-o-monoidal}]
	By construction, $\Day_{\cC,\cD}^\otimes$ is an operad, so \cite[Proposition~2.1.2.12]{HA} shows that we only have to check that $\Day_{\cC,\cD}^\otimes \to \cO^\otimes$ is a cocartesian fibration.
    As in \cite[Section~2.2.6]{HA}, the crucial part of the argument lies in identifying mapping anima in $\Day_{\cC,\cD}^\otimes$.
    
    We require some notation.
    Let $\pi \colon \Day_{\cC,\cD}^\otimes \to \cO^\otimes$ and $u \colon \Day_{\cC,\cD}^\otimes \to (\Ar^\oplax)^\times$ denote the projection functors, and abbreviate $\cX^\times := \catinf^\times \times_{\Comm^\otimes} \catinf^\times$.
    Note that then $\cX = \catinf \times \catinf$.
    Let $\phi \colon x \to y$ be a morphism in $\cO^\otimes$, let $F \in \pi^{-1}(x)$ and $G \in \pi^{-1}(y)$, and denote by $\Hom_{\Day_{\cC,\cD}^\otimes}^\phi(F,G)$ the anima of morphisms lying over $\phi$.
    Since $\Day_{\cC,\cD}^\otimes$ is defined as a pullback, we have a natural pullback square
    \[\begin{tikzcd}
		\Hom_{\Day_{\cC,\cD}^\otimes}^\phi(F,G)\ar[r]\ar[d] & \Hom_{(\Ar^\oplax)^\times}(uF,uG)\ar[d] \\
		*\ar[r, "{(\cC,\cD) \circ \phi}"] & \Hom_{\cX^\times}((s,t)F,(s,t)G)
	\end{tikzcd}\]
	In particular, the objects $sF$, $tF$, $sG$ and $tG$ are identified with $\cC(x)$, $\cD(x)$, $\cC(y)$ and $\cD(y)$, respectively.
	Denoting by $\alpha \colon \langle k \rangle \to \langle l \rangle$ the image of $\phi$ in $\Comm^\otimes$, the anima $\Hom_{\Day_{\cC,\cD}^\otimes}^\phi(F,G)$ sits in a natural fibre square
	\[\begin{tikzcd}
		\Hom_{\Day_{\cC,\cD}^\otimes}^\phi(F,G)\ar[r]\ar[d] & \Hom_{(\Ar^\oplax)^\times}^\alpha(uF,uG)\ar[d] \\
		*\ar[r, "{(\cC,\cD) \circ \phi}"] & \Hom_{\cX^\times}^\alpha((s,t)F,(s,t)G)
	\end{tikzcd}\]
	Since both $(s,t) \colon (\Ar^\oplax)^\times \to \cX^\times$ and $\cX^\times \to \Comm^\otimes$ are cocartesian fibrations, the right vertical map is identified with the map
	\[ \Hom_{(\Ar^\oplax)^\times_{\langle l \rangle}}^{\id_{\langle l \rangle}}(\alpha_!(uF),uG) \to \Hom_{\cX^\times_{\langle l \rangle}}^{\id_{\langle l \rangle}}(\alpha_!((s,t)F),(s,t)G) \]
	induced by $(s,t)$.
	Write $F = F_1 \boxtimes \ldots \boxtimes F_k$ and $G = G_1 \boxtimes \ldots \boxtimes G_l$.
	As both $\Ar^\oplax$ and $\cX$ carry the cartesian symmetric monoidal structure, this map is in turn identified with the map
	\begin{equation}\label{eq:day-o-monoidal-1}
		\prod_{j=1}^l \Hom_{\Ar^\oplax}\Big( \prod_{i \in \alpha^{-1}(j)} F_i, G_j \Big) \to \prod_{j=1}^l \Hom_{\cX}\Big( \prod_{i \in \alpha^{-1}(j)} (s,t)F_i, (s,t)G_j \Big)
	\end{equation}
	induced by $(s,t)$.
    Consequently, it suffices to consider the case that $\alpha \colon \langle k \rangle \to \langle 1 \rangle$ is an active morphism so that $\phi \in \mathrm{Mul}_\cO(\{x_i\}_i,y)$.
    
    With respect to the given identifications, the base point $(\cC,\cD) \circ \phi$ now becomes the point in
    \begin{align*}
     \Hom_{\cX}&\big(\prod_{i=1}^k (\cC(x_i), \cD(x_i)), (\cC(y),\cD(y)\big) \\
     &\simeq \Hom_\catinf\big(\prod_{i=1}^k \cC(x_i),\cC(y)\big) \times \Hom_\catinf\big(\prod_{i=1}^k \cD(x_i),\cD(y)\big)
    \end{align*}
    corresponding to the pair of multiplication functors $(\otimes_\phi^\cC,\otimes_\phi^\cD)$ of $\cC$ and $\cD$.

    By \cref{thm:main} and \cref{rem:aropl-straightened}, the fibre of \eqref{eq:day-o-monoidal-1} is identified with the anima of natural transformations
    \[ \Nat\big(\otimes_\phi^\cD \circ \prod_{i=1}^k \St(F_i), \St(G) \circ \otimes_\phi^\cC\big).\]
    Fix now $\phi \colon x \to y$ and $F \in \pi^{-1}(x)$.
    As before, let $\alpha \colon \langle k \rangle \to \langle l \rangle$ be the image of $\phi$ in $\Comm^\otimes$ and let $F = F_1 \boxtimes \ldots \boxtimes F_k$ be the canonical decomposition of $F$ with $F_i \in \Ar^\oplax$.
    Denote by $\phi_j \in \mathrm{Mul}_\cO(\{x_i\}_{i \in \alpha^{-1}(j)},y_j)$ the active morphisms determined by $\phi$ and $\alpha$.
    Using assumption~\eqref{prop:day-o-monoidal-1} and the pointwise formula for left Kan extensions, the composite
    \[ \prod_{i \in \alpha^{-1}(j)} \cC(x_i) \xrightarrow{\prod_i \St(F_i)} \prod_{i \in \alpha^{-1}(j)} \cD(x_i) \xrightarrow{\otimes_{\phi_j}^\cD} \cD(y_j) \]
    admits a left Kan extension $G_j$ along $\otimes_{\phi_j}^\cC \colon \prod_{i \in \alpha^{-1}(j)} \cC(x_i) \to \cC(y_j)$ for each $j \in \langle l \rangle$.
    As we have seen, the unit transformations
    \[ \eta_j \colon \otimes_{\phi_j}^\cD \circ \prod_{i \in \alpha^{-1}(j)} \St(F_i) \Rightarrow \St(G_j) \circ \otimes_{\phi_j}^\cC \]
    determine a point $\eta \in \Hom_{\Day_{\cC,\cD}^\otimes}^\phi(F,G)$, where $G := G_1 \boxtimes \ldots \boxtimes G_l$.
    We claim that $\eta$ is a cocartesian lift of $\phi$.

    This amounts to checking that for each $H \in \Day_{\cC,\cD}^\otimes$, the induced commutative square
    \[\begin{tikzcd}
        \Hom_{\Day_{\cC,\cD}^\otimes}(G,H)\ar[r, "- \circ \eta"]\ar[d] & \Hom_{\Day_{\cC,\cD}^\otimes}(F,H)\ar[d] \\
        \Hom_{\cO^\otimes}(y,z)\ar[r, "- \circ \phi"] & \Hom_{\cO^\otimes}(x,z)
    \end{tikzcd}\]
    is a pullback, where we set $z := \pi(H)$.
    This is equivalent to the assertion that for each $\psi \in \Hom_{\cO^\otimes}(y,z)$, the induced map on vertical fibres
    \begin{equation}\label{eq:day-o-monoidal-2}
        - \circ \eta \colon \Hom_{\Day_{\cC,\cD}^\otimes}^{\psi}(G,H) \to \Hom_{\Day_{\cC,\cD}^\otimes}^{\psi\phi}(F,H)    
    \end{equation}
    is an equivalence.
    Letting $\beta \colon \langle l \rangle \to \langle n \rangle$ denote the image of $\psi$ in $\Comm^\otimes$, the preliminary discussion and 
    \cref{thm:main} identify this map with the product of the maps
    \begin{align*}
        \Nat\Big(&\otimes_{\psi_m}^\cD \circ \prod_{j \in \beta^{-1}(m)} \big(\St(G_j) \circ \otimes_{\phi_j}^\cC\big), \St(H_m) \circ \otimes_{\psi_m}^\cC \circ \prod_{j \in \beta^{-1}(m)} \otimes_{\phi_j}^\cC\Big) \\
        &\xrightarrow{\eta_j^*} \Nat\Big( \otimes_{\psi_m}^\cD \circ \prod_{j \in \beta^{-1}(m)} \big(\otimes_{\phi_j}^\cD \circ \prod_{i \in \alpha^{-1}(j)} \St(F_i)\big), \St(H_m) \circ \otimes_{(\psi\phi)_m}^\cC\Big) \\
        &\simeq \Nat\Big( \otimes_{(\phi\psi)_m}^\cD \circ \prod_{i \in (\beta\alpha)^{-1}(m)} \St(F_i), \St(H_m) \circ \otimes_{(\psi\phi)_m}^\cC\Big).
    \end{align*}
    Using assumption~\eqref{prop:day-o-monoidal-2}, the pointwise formula for left Kan extensions implies that the transformation
    \[ \otimes_{\psi_m}^\cD \circ \prod_{j \in \beta^{-1}(m)} \big(\otimes_{\phi_j}^\cD \circ \prod_{i \in \alpha^{-1}(j)} \St(F_i)\big) \Rightarrow \otimes_{\psi_m}^\cD \circ \prod_{j \in \beta^{-1}(m)} \big(\St(G_j) \circ \otimes_{\phi_j}^\cC\big) \]
    induced by $\eta_j$ also exhibits $\otimes_{\psi_m}^\cD \circ \prod_{j \in \beta^{-1}(m)} \St(G_j)$ as a left Kan extension, so \eqref{eq:day-o-monoidal-2} is an equivalence.
\end{proof}

\begin{rem}
    In the situation of \cref{prop:day-o-monoidal}, assume that $\cO$ is a symmetric monoidal category, so that $\cC$ and $\cD$ correspond to lax symmetric monoidal functors $\cO \to \catinf$.
    Unwinding the proof of \cref{prop:day-o-monoidal}, one obtains the following description of the lax symmetric monoidal functor $\cO \to \catinf$ given by the straightening of $\Day_{\cC,\cD}^\otimes \to \cO^\otimes$:
    \begin{enumerate}
        \item The underlying functor $\cO \to \catinf$ sends $x \in \cO$ to $\Fun(\cC(x),\cD(x))$ and a morphism $f \colon x \to x'$ to the composite
        \[ \Fun(\cC(x),\cD(x)) \xrightarrow{f \circ -} \Fun(\cC(x),\cD(x')) \xrightarrow{f_!} \Fun(\cC(x'),\cD(x')), \]
        where $f_!$ denotes the left Kan extension functor.
        \item For $x, x' \in \cO$, the lax monoidal structure map is given by the composite
        \begin{align*}
            \Fun(\cC(x),\cD(x)) \times \Fun(\cC(x'),\cD(x'))
            &\to \Fun(\cC(x) \times \cC(x'),\cD(x) \times \cD(x')) \\
            &\xrightarrow{\otimes_\cD \circ -} \Fun(\cC(x) \times \cC(x'), \cD(x \otimes x')) \\
            &\xrightarrow{(\otimes_\cC)_!} \Fun(\cC(x \otimes x'),\cD(x \otimes x')),
        \end{align*}
        where the first arrow arises from the lax symmetric monoidal structure on $\Fun \colon \catinf^\op \times \catinf \to \catinf$, and $\otimes_\cC$ and $\otimes_\cD$ denote the tensor operations in $\cC$ and $\cD$, respectively.
        \item The structure map associated to the monoidal unit is given by the object in $\Fun(\cC(\mathbf{1}_\cO),\cD(\mathbf{1}_\cO))$ which arises as the left Kan extension of $* \xrightarrow{\mathbf{1}_\cD} \cD(\mathbf{1}_\cO)$ along $* \xrightarrow{\mathbf{1}_\cC} \cC(\mathbf{1}_\cO)$.
    \end{enumerate}
\end{rem}

\bibliographystyle{alpha}
\bibliography{day}

\newcommand{\etalchar}[1]{$^{#1}$}
\begin{thebibliography}{HHLN23b}

\bibitem[CDH{\etalchar{+}}23]{thenine-1}
B.~Calm{\`e}s, E.~Dotto, Y.~Harpaz, F.~Hebestreit, M.~Land, K.~Moi, D.~Nardin,
  T.~Nikolaus, and W.~Steimle.
\newblock Hermitian {K}-theory for stable {{\(\infty\)}}-categories. {I}:
  {Foundations}.
\newblock {\em Sel. Math., New Ser.}, 29(1):269, 2023.
\newblock Id/No 10.

\bibitem[GHN17]{ghn:free-fibrations}
D.~Gepner, R.~Haugseng, and T.~Nikolaus.
\newblock Lax colimits and free fibrations in {{\(\infty\)}}-categories.
\newblock {\em Doc. Math.}, 22:1225--1266, 2017.

\bibitem[Gla16]{glasman:day}
S.~Glasman.
\newblock Day convolution for {{\(\infty\)}}-categories.
\newblock {\em Math. Res. Lett.}, 23(5):1369--1385, 2016.

\bibitem[HHLN23a]{hhln:lax-adjunctions}
R.~Haugseng, F.~Hebestreit, S.~Linskens, and J.~Nuiten.
\newblock Lax monoidal adjunctions, two-variable fibrations and the calculus of
  mates.
\newblock {\em Proc. Lond. Math. Soc. (3)}, 127(4):889--957, 2023.

\bibitem[HHLN23b]{hhln:two-variable-fibrations}
R.~Haugseng, F.~Hebestreit, S.~Linskens, and J.~Nuiten.
\newblock Two-variable fibrations, factorisation systems and
  {{\(\infty\)}}-categories of spans.
\newblock {\em Forum Math. Sigma}, 11:70, 2023.
\newblock Id/No e111.

\bibitem[Hin15]{hinich:rectification}
V.~Hinich.
\newblock Rectification of algebras and modules.
\newblock {\em Doc. Math.}, 20:879--926, 2015.

\bibitem[Hin20]{hinich:enriched-yoneda}
V.~Hinich.
\newblock Yoneda lemma for enriched {{\(\infty\)}}-categories.
\newblock {\em Adv. Math.}, 367:119, 2020.
\newblock Id/No 107129.

\bibitem[HMS22]{hms:shifted}
R.~Haugseng, V.~Melani, and P.~Safronov.
\newblock Shifted coisotropic correspondences.
\newblock {\em J. Inst. Math. Jussieu}, 21(3):785--849, 2022.

\bibitem[Lur]{HA}
J.~Lurie.
\newblock Higher {A}lgebra.
\newblock Available on the author's homepage:
  \href{https://www.math.ias.edu/~lurie/papers/HA.pdf}{https://www.math.ias.edu/{\textasciitilde}lurie/papers/HA.pdf}.

\bibitem[Lur09]{HTT}
J.~Lurie.
\newblock {\em {Higher topos theory}}, volume 170.
\newblock Princeton, NJ: Princeton University Press, 2009.

\end{thebibliography}

\end{document}